\documentclass[11pt,reqno]{amsart}

\usepackage[a4paper,margin=1.08in]{geometry}
\usepackage{amsmath,amssymb,amsthm,mathtools}
\usepackage{microtype}
\usepackage{enumitem}
\usepackage{booktabs}
\usepackage[hidelinks]{hyperref}
\hypersetup{
	pdftitle={Generalized Erdos--Ginzburg--Ziv Constants in Rank Three and Elementary Rank Four},
	pdfsubject={Exact generalized EGZ constants in rank three and elementary rank four},
	pdfkeywords={zero-sum sequence, generalized EGZ constant, finite abelian group}
}

\numberwithin{equation}{section}

\newtheorem{theorem}{Theorem}[section]
\newtheorem{proposition}[theorem]{Proposition}
\newtheorem{lemma}[theorem]{Lemma}
\newtheorem{conjecture}[theorem]{Conjecture}
\newtheorem{corollary}[theorem]{Corollary}

\theoremstyle{definition}

\theoremstyle{remark}

\newcommand{\F}{\mathbb F}

\newcommand{\Span}{\operatorname{span}}

\newcommand{\bdot}{\boldsymbol{\cdot}}

\title[Generalized EGZ Constants]{Generalized Erd\H{o}s--Ginzburg--Ziv Constants for finite abelian groups of rank three}

\date{}
\author{Hanbin Zhang}
\address{School of Mathematics (Zhuhai), Sun Yat-sen University, Zhuhai 519082, Guangdong, P.R. China}
\email{zhanghb68@mail.sysu.edu.cn}

\author{Shiwen Zhang}
\address{Department of Mathematics, Southern University of Science and Technology, Shenzhen 518055, Guangdong, P.R. China}
\email{12431019@mail.sustech.edu.cn}

\begin{document}
	
	\begin{abstract}
		Let $G$ be a finite abelian group with exponent $\exp(G)$. For positive integer $k$, the generalized Erd\H{o}s--Ginzburg--Ziv constant, denoted by $\mathsf s_{k\exp(G)}(G)$
		is the least integer $\ell$ such that every sequence of
		length at least $\ell$ over $G$ has a zero-sum subsequence of length $k\exp(G)$.
		In this paper, we prove that, for any $n\ge 1$,
		\[
		\mathsf s_{3\cdot 2^n}(C_{2^n}^3)=6\cdot 2^n-3
		\quad\text{and}\quad
		\mathsf s_{4\cdot 3^n}(C_{3^n}^3)=7\cdot 3^n-3.
		\]
		This completes all the remaining cases in the study of $\mathsf s_{kq}(C_q^3)$, where $q$ is a prime power and $k\ge 3$. As a consequence, for every $k\ge 3$, we have 
		\[
		\mathsf s_{kn}(C_n^3)=(k+3)n+O(\frac{n}{\log n}).
		\]
		This verifies a conjecture of Kubertin asymptotically. Our approach combines the previous group ring techniques and a recent powerful generalization of the Chevalley-Warning and Ax-Katz theorems introduced by Grynkiewicz
		
	\end{abstract}
	
	\maketitle
	
	\section{Introduction}

	Let $G$ be a finite abelian group, written additively, with exponent
	$\exp(G)=n$.  Let $S=g_1\bdot\ldots\bdot g_k$ be a sequence over $G$ (unordered,  repetition is allowed, and $g_i\in G$ for $1\le i\le k$), where $k$ is called the length of $S$. We call $S$ a zero-sum sequence if $\sum^k_{i=1}g_i=0$, the identity of $G$. As an active branch of additive combinatorics,
	zero-sum theory studies various properties on zero-sum subsequences (existence and structure); see \cite{GaoGeroldinger}. In 1961, Erd\H{o}s, Ginzburg and Ziv \cite{EGZ} proved that from every sequence $S$ over $C_n$ of length $2n-1$, we can always find a zero-sum subsequence $T$ of length $n$ (see \cite{ADZ} for other different proofs). This classical result is one of the starting point of zero-sum theory. For any positive integer $k$, the $k$-th Erd\H{o}s-Ginzburg-Ziv constant $\mathsf s_{kn}(G)$ is defined as the smallest positive integer $t$ such that every sequence $S$ over $G$ of length at least $t$ has a zero-sum subsequence of length $kn$. In particular, for $k=1$, $\mathsf s(G):=\mathsf s_n(G)$ is called the Erd\H{o}s-Ginzburg-Ziv constant. The invariants $\mathsf s_{kn}(G)$ have been extensively studied (see, e.g., \cite{GHPS,GHZ,GHP,GT,Grynkiewicz,HZ18,HZ19,HZ192,He,Kubertin,Reiher,Zhang}) and are closely related to coding theory \cite{Sidorenko} and graph theory \cite{Sidorenko1}. Another important objects in zero-sum theory is the Davenport constant $\mathsf D(G)$, which is defined as the smallest positive integer $t$ such that every sequence $S$ over $G$ of length at least $t$ has a nonempty zero-sum subsequence. We refer to \cite{HanZhang2026,Wang2026} for some recent studies and the references therein.
	
	If $T$ is a zero-sum free sequence of length $\mathsf D(G)-1$, then
	\(
	0^{[kn-1]}T
	\)
	has length $kn+\mathsf D(G)-2$ and has no zero-sum subsequence of length $kn$.
	Therefore
	\begin{equation}\label{eq:standard-lower}
		\mathsf s_{kn}(G)\ge kn+\mathsf D(G)-1.
	\end{equation}
	In 1996, Gao \cite{Gao3} proved that $\mathsf s_{kn}(G)= kn+\mathsf D(G)-1$, provided that $kn\ge|G|$. In \cite{GT}, Gao and Thangadurai showed that if $kn<\mathsf D(G)$, then $\mathsf s_{kn}(G)> kn+\mathsf D(G)-1$. Moreover, they define $\ell(G)$ be the smallest integer $t$ such that
	$$\mathsf s_{kn}(G)=kn+\mathsf D(G)-1$$
	holds for every $k\ge t$. From the above we know that $\frac{\mathsf{D}(G)}{n}\le \ell(G)\le \frac{|G|}{n}.$
	Recently, Gao, Han, Peng and Sun (\cite{GHPS}, Conjecture 4.7)  conjectured that $\ell(G)=\lceil\frac{\mathsf D(G)}{n}\rceil$, for any finite abelian group $G$ with $\exp(G)=n$. For a prime power $q$, Olson \cite{Olson} proved that
	\(
	\mathsf D(C_q^r)=r(q-1)+1.
	\)
	So \eqref{eq:standard-lower} becomes
	\begin{equation}\label{eq:homocyclic-lower}
		\mathsf s_{kq}(C_q^r)\ge (k+r)q-r.
	\end{equation}
	Kubertin \cite{Kubertin} proposed the following conjecture. 
	\begin{conjecture}\label{Kub}
		Let $q$ be a prime power, then
		$$\mathsf s_{kq}(C_q^r)= (k+r)q-r$$
		holds for $k\ge r$.
	\end{conjecture}
	There are several important progress towards this conjecture (see, e.g., \cite{GHPS,GHP,GT,Grynkiewicz,HZ18,HZ19,HZ192,He,Kubertin,Zhang}), but it remains widely open for abelian groups of rank at least 3. In particular, for abelian groups of rank 3, previous techniques involving group rings do not work for the case $q=2^n$ and $3^n$.
	
	In this paper, we complete all the remaining cases for Conjecture \ref{Kub} for $r=3$. Our approach combines the previous group ring techniques and a recent powerful generalization of the Chevalley-Warning and Ax-Katz theorems introduced by Grynkiewicz \cite{Grynkiewicz}.
	
	\begin{theorem}\label{thm:intro-dyadic}
		For every $n\ge 1$,
		\[
		\mathsf s_{3\cdot 2^n}(C_{2^n}^3)=6\cdot 2^n-3.
		\]
	\end{theorem}
	
	
	\begin{theorem}\label{thm:intro-ternary}
		For every $n\ge 1$, 
		\[
		\mathsf s_{4\cdot 3^n}(C_{3^n}^3)=7\cdot 3^n-3.
		\]
	\end{theorem}
	
	As a consequence, we have the following theorem.
	
	\begin{theorem}\label{thm:intro-rank-three}
		Let $q$ be a prime power.  For every $k\ge 3$,
		\[
		\mathsf s_{kq}(C_q^3)=(k+3)q-3.
		\]
		Moreover, for every $k\ge 3$, we have
		\[
		\mathsf s_{kn}(C_n^3)=(k+3)n+O(\frac{n}{\log n}).
		\]
	\end{theorem}

	The paper is organized as follows.  Section~\ref{sec:prelim} introduces some preliminary results.  Section~\ref{sec:dyadic} proves Theorem~\ref{thm:intro-dyadic}.  Section~\ref{sec:ternary-rank-three} proves Theorem~\ref{thm:intro-ternary} and completes Theorem~\ref{thm:intro-rank-three}.

	\section{Notation and preliminary results}\label{sec:prelim}
	
	A sequence over $G$ is written multiplicatively,
	\[
	S=g_1\boldsymbol\cdot\ldots\boldsymbol\cdot g_\ell.
	\]
	Its length and sum are
	\[
	|S|=\ell,
	\qquad
	\sigma(S)=g_1+\cdots+g_\ell.
	\]
	A subsequence is always indexed.  Thus repeated group elements give distinct
	choices.  For $t\ge0$ and $u\in G$, let
	\[
	N_t^u(S)=\#\{T\mid S: |T|=t,\ \sigma(T)=u\},
	\qquad
	N_t(S)=N_t^0(S),
	\]
	and put
	\[
	\Sigma_t(S)=\{u\in G:N_t^u(S)>0\}.
	\]
	We use $g^{[m]}$ for a sequence consisting of $m$ copies of $g$.  If
	$\sigma(S)=0$, taking the complementary parts gives
	\begin{equation}\label{eq:complement}
		N_t^u(S)=N_{|S|-t}^{-u}(S),
		\qquad
		N_t(S)=N_{|S|-t}(S).
	\end{equation}
	
	We first give a basic lemma, which follows from Olson's theorem~\cite{Olson}.  The strategy in its proof will also be used repeatedly.
	
	\begin{lemma}\label{lem:olson-affine}
		Let $G$ be a finite abelian $p$-group of exponent $q$, and let $S$ be a
		sequence with $|S|\ge \mathsf D(G)+q-1$.  Then
		\begin{equation}\label{eq:olson-multiple}
			1+\sum_{j=1}^{\lfloor |S|/q\rfloor}(-1)^jN_{jq}(S)\equiv0\pmod p.
		\end{equation}
	\end{lemma}
	
	\begin{proof}
		Map every term $g\in G$ to $(g,1)\in G\oplus C_q$.  A subsequence is
		zero-sum in the enlarged group if and only if its sum in $G$ is zero and its
		length is divisible by $q$.  Since
		\[
		\mathsf D(G\oplus C_q)=\mathsf D(G)+q-1,
		\]
		Olson's congruence applied in $G\oplus C_q$ gives
		\eqref{eq:olson-multiple}.
	\end{proof}
	
	The following is the form of the weighted
	Ax--Katz theorem needed in this paper.
	
	\begin{theorem}\label{thm:weighted-moments}
		Let $G$ be a nontrivial finite abelian $p$-group of exponent $q$, and let
		$\mathsf D(G)$.  Let $m\ge0$, $t\ge1$, and $0\le\alpha<q$.  If
		\[
		|S|\ge m\frac{(p-1)q}{p}+tq-1+\mathsf D(G),
		\]
		then, for $0\le i<t$,
		\begin{equation}\label{eq:weighted-moments}
			\sum_{j\ge0}(p-1)^{jq+\alpha}j^iN_{jq+\alpha}(S)
			\equiv0\pmod {p^{m+1}}.
		\end{equation}
		The same congruence holds for $N_{jq+\alpha}^u(S)$ with any fixed target
		$u\in G$.
	\end{theorem}
	
	The theorem is a direct specialization of the weighted
	Chevalley--Warning--Ax--Katz theory developed in~\cite{GGS,Grynkiewicz}.  We shall apply it only with
	$\alpha=0$ and with small values of $m$ and $t$.
	
	The integer-valued version will be useful twice.  If $W$ is an
	integer-valued polynomial of degree at most $t-1$, the same argument gives a
	weighted congruence with $W(j)$ in place of $j^i$.  In particular, taking
	$W(X)=\binom{X-1}{k-1}$ yields the following result.
	
	\begin{lemma}\label{lem:binomial-selector}
		Let $G$ be a finite abelian $p$-group of exponent $q$, and let $\mathsf D(G)$.
		If $|S|\ge kq+\mathsf D(G)-1$, then
		\begin{equation}\label{eq:binomial-selector}
			\sum_{j\ge0}(p-1)^{jq}\binom{j-1}{k-1}N_{jq}(S)\equiv0\pmod p.
		\end{equation}
		Here $\binom{-1}{k-1}=(-1)^{k-1}$.
	\end{lemma}
	
	\begin{proof}
		Apply the integer-valued form of Theorem~\ref{thm:weighted-moments} with
		$W(X)=\binom{X-1}{k-1}$.  This polynomial has degree $k-1$, so the
		required length is $kq+\mathsf D(G)-1$, and the resulting weighted count is exactly
		\eqref{eq:binomial-selector}.
	\end{proof}
	
	We use Lucas' theorem~\cite{Lucas} in its standard form.
	
	\begin{lemma}[Lucas]\label{lem:lucas}
		Let $p$ be prime, and write
		$a=\sum_i a_ip^i$ and $b=\sum_i b_ip^i$ with $0\le a_i,b_i<p$.  Then
		\[
		\binom ab\equiv\prod_i\binom{a_i}{b_i}\pmod p.
		\]
	\end{lemma}

	We recall some previously known results that enter the proofs.
	
	\begin{theorem}\label{thm:known-rank-three}
		Let $q=p^n$ be a prime power.
		\begin{enumerate}[label=\textup{(\roman*)}]
			\item If $p=2$, then
			\[
			\mathsf s_{kq}(C_q^3)=(k+3)q-3\qquad(k\ge4).
			\]
			\item If $p=3$, then the same formula holds for $k=3,5$ and for every
			$k\ge6$.
			\item If $p\ge5$, then the same formula holds for every $k\ge3$.
		\end{enumerate}
	\end{theorem}
	
	These statements follow from Han--Zhang~\cite{HZ19}, Zhang~\cite{Zhang}, and the
	general large-prime $p$-group theory~\cite{GHP,He,Grynkiewicz}.  We shall prove the two missing cases
	$(p,k)=(2,3)$ and $(3,4)$.

	We also need the following binary theorem of Sidorenko~\cite{Sidorenko}.
	
	\begin{theorem}[Sidorenko]\label{thm:sidorenko}
		If $d<2m$, then
		\[
		\mathsf s_{2m}(C_2^d)=2m+d.
		\]
		In particular,
		\[
		\mathsf s_4(C_2^3)=7,
		\qquad
		\mathsf s_6(C_2^3)=9.
		\]
	\end{theorem}
	
	A set-valued consequence of the weighted theorem will be used in the ternary
	rank-three proof.
	
	\begin{lemma}\label{lem:ternary-four-length}
		Let $q$ be a power of $3$.  Every sequence over $C_q^3$ of length at least
		$\frac{17}{3}q-3$ contains a zero-sum subsequence whose length belongs to
		\[
		\{q,2q,4q,5q\}.
		\]
	\end{lemma}
	
	This is the specialization of Grynkiewicz's set-valued theorem~\cite{Grynkiewicz}
	obtained from the multiplier set $\{1,2,4,5\}$.
	
	Let $A$ be a sequence over $C_{2q}^3$, where $q$ is a power of $2$, and let
	\[
	\pi:C_{2q}^3\longrightarrow C_2^3
	\]
	be reduction modulo $2$.  Pair terms with equal image under $\pi$.  For each
	pair $P_i=x_i\boldsymbol\cdot y_i$ write
	\[
	x_i+y_i=2b_i,
	\qquad b_i\in C_q^3.
	\]
	The sequence $B=b_1\boldsymbol\cdot\ldots\boldsymbol\cdot b_m$ is the divided
	pair-sum sequence.  After maximal pairing, at most one term remains in each
	quotient class; their residue set is denoted by $O\subset C_2^3$.  Thus
	\begin{equation}\label{eq:pairing-length}
		|A|=2|B|+|O|.
	\end{equation}
	If $A$ has total sum zero, then $\sum_{x\in O}x=0$ in $C_2^3$.
	When $q=2$, two equal divided pair sums produce a four-term zero sum:
	\begin{equation}\label{eq:equal-pair-sums}
		b_i=b_j\quad\Longrightarrow\quad \sigma(P_iP_j)=0.
	\end{equation}
	Consequently, in a four-zero-free sequence over $C_4^3$, the sequence $B$ is
	squarefree and hence may be identified with a subset of $C_2^3$.
	
	We collect the elementary facts about $V=C_2^3=\F_2^3$ that will be used.
	For a subset $X\subset V$, write $\Sigma_j(X)$ for the set of sums of its
	$j$-element subsets.
	
	\begin{lemma}\label{lem:f2-geometry}
		Let $X\subset V$.
		\begin{enumerate}[label=\textup{(\roman*)}]
			\item If $|X|=5$, then
			\[
			\Sigma_2(X)=V\setminus\{0\},
			\qquad
			\Sigma_3(X)=V\setminus\{\sigma(X)\}.
			\]
			Moreover, every element of $V$ is a sum of either two or four distinct
			elements of $X$.
			\item If $|X|=6$, then $\Sigma_3(X)=V$,
			$\Sigma_2(X)=V\setminus\{0\}$, and $|\Sigma_4(X)|=7$.
			\item If $|X|=7$, then $\Sigma_3(X)=\Sigma_4(X)=V$.
			\item For $X=V$, one has
			$\Sigma_3(V)=\Sigma_4(V)=\Sigma_5(V)=V$.
			\item If $|X|=4$ and $\sigma(X)\ne0$, then the sums of the even subsets of
			$X$, of sizes $0,2,4$, exhaust $V$; the same is true of the odd subsets, of
			sizes $1,3$.
		\end{enumerate}
	\end{lemma}
	
	\begin{proof}
		For $a\ne0$, the four pairs $\{x,x+a\}$ partition $V$.  Hence every subset of
		size at least five contains a pair of sum $a$, proving the assertions about
		$\Sigma_2$.  Complementation inside $X$ gives
		\[
		\Sigma_{|X|-j}(X)=\sigma(X)+\Sigma_j(X),
		\]
		which yields the statements about three- and four-element sumsets.  The
		remaining full-sumset assertions follow by deleting one or two points from
		$V$.
		
		For the last assertion, the empty set gives $0$, the full set gives
		$\sigma(X)$, and the six two-element sums are distinct and avoid these two
		values.  Indeed, equality of complementary pair sums would imply
		$\sigma(X)=0$.  Thus the eight even-subset sums are all of $V$, and
		complementation gives the odd-subset statement.
	\end{proof}
	
	We shall also use a simple affine-plane obstruction, first at exponent four
	and later at every dyadic exponent.
	
	\begin{lemma}\label{lem:affine-plane-4}
		Every sequence of length $17$ over $C_4^2\oplus C_2$ contains a zero-sum
		subsequence of length $4$.  Consequently, if a sequence of length at least
		$17$ over $C_4^3$ has all quotient residues modulo $2$ in one affine plane of
		$V$, then it contains a four-term zero sum.
	\end{lemma}
	
	\begin{proof}
		Reduce a $17$-term sequence over $C_4^2\oplus C_2$ modulo twice the group.
		There are eight quotient classes.  Since the length is odd, maximal pairing
		leaves at most seven unpaired terms and therefore produces at least five
		pairs.  Their divided sums lie in $C_2^2$, which has four elements.  Two pair
		sums coincide, and the corresponding four original terms have sum zero.
		
		For the consequence, translate the affine plane to a linear plane.  This
		translation preserves four-term zero sums, and the full inverse image of the
		linear plane is isomorphic to $C_4^2\oplus C_2$.
	\end{proof}
	
	\section{Proof of Theorem \ref{thm:intro-dyadic}}\label{sec:dyadic}
	In this section, let $q$ be a power of $2$ and $V=C_2^3=\F_2^3$.
	We first prove some results on $C_4^3$ and then propagate it by a simultaneous quotient induction.
	
	\begin{theorem}\label{thm:c4-rank3-threshold}
		One has
		\[
		\mathsf s_{12}(C_4^3)=21.
		\]
	\end{theorem}
	
	\begin{proof}
		The lower bound is
		\[
		12+\mathsf D(C_4^3)-1=12+10-1=21.
		\]
		Let $S$ be a sequence of length $21$ with no twelve-term zero sum.  Since
		$21\equiv1\pmod4$, translation allows us to assume that $\sigma(S)=0$.
		
		We first show that $N_4(S)=0$.  Suppose that $Q_1\mid S$ is a four-term zero
		sum.  The complement has length $17$.  Choose any $13$ terms from it and
		apply Lemma~\ref{lem:olson-affine}.  The chosen sequence contains a zero-sum
		subsequence of length $4$, $8$, or $12$.  The last two alternatives give a
		forbidden twelve-term zero sum after adjoining $Q_1$, or directly.  Hence
		there is a second four-term zero sum $Q_2$ disjoint from $Q_1$.  The remaining
		sequence has length $13$, and one more application of
		Lemma~\ref{lem:olson-affine} gives a zero sum of length $4$, $8$, or $12$.
		Each alternative now produces a twelve-term zero sum in $S$, a contradiction.
		Thus
		\begin{equation}\label{eq:c4-no4}
			N_4(S)=0.
		\end{equation}
		
		Pair equal residues modulo $2$.  By \eqref{eq:equal-pair-sums}, the divided
		pair sums form a subset $B\subset V$.  Let $O$ be the unpaired residue set.
		Since $|S|$ is odd and $\sigma(S)=0$, we have
		\[
		|O|\in\{1,3,5,7\},
		\qquad
		\sigma(O)=0.
		\]
		If $|O|=1$ or $3$, then $|B|=10$ or $9$.  By
		Theorem~\ref{thm:sidorenko}, six terms of $B$ have sum zero, and the
		corresponding six pairs give a twelve-term zero sum in $S$.
		
		Suppose $|O|=5$.  The equality $\sigma(O)=0$ forces $0\in O$, and the other
		four residues have sum zero.  Since $|B|=8$, one has $B=V$.  The four lifted
		unpaired terms have sum $2c$ for some $c\in V$.  By
		Lemma~\ref{lem:f2-geometry}(iv), four terms of $B$ have sum $-c$.  The
		corresponding four pairs, together with the four unpaired terms, form a
		zero-sum subsequence of length twelve.
		
		Finally, if $|O|=7$, then $O=V\setminus\{0\}$.  Choose a four-element
		zero-sum subset of $O$.  The divided pair sums form a seven-set, whose
		four-term sumset is all of $V$ by Lemma~\ref{lem:f2-geometry}(iii).  Again the
		carry can be absorbed by four pairs, producing a twelve-term zero sum.
	\end{proof}
	
	\begin{theorem}\label{thm:c4-layer10}
		Let $A$ be a sequence over $C_4^3$ such that
		\[
		|A|\in\{19,20\},
		\qquad
		N_4(A)=N_{12}(A)=0.
		\]
		Then
		\[
		\Sigma_{10}(A)=C_4^3.
		\]
	\end{theorem}
	
	\begin{proof}
		Pair equal residues modulo $2$, let $B$ be the divided pair-sum set, and let
		$O$ be the unpaired residue set.  By \eqref{eq:equal-pair-sums}, $B$ is
		squarefree.  Fix a target $x\in C_4^3$, and write $\bar x$ for its image in
		$V$.
		
		Assume first that $|A|=20$.  Then $|O|\in\{4,6,8\}$.
		
		If $|O|=8$, then $O=V$ and $|B|=6$.  Choose four unpaired residues with sum
		$\bar x$.  Their lifted sum differs from $x$ by $2c$ for some $c\in V$.
		Since a six-set has full three-term sumset, three divided pair sums can be
		chosen with sum $-c$.  The resulting ten terms have sum $x$.
		
		If $|O|=6$, then $|B|=7$.  When $\bar x\ne0$, choose two unpaired residues
		with sum $\bar x$; when $\bar x=0$, choose four unpaired residues with sum
		zero.  Use respectively four or three complete pairs to absorb the carry.
		The required restricted sumsets of a seven-set are full by
		Lemma~\ref{lem:f2-geometry}(iii).
		
		Suppose $|O|=4$.  Then $B=V$.  Put $t=\sigma(O)$.  If $t=0$, the four lifted
		unpaired terms have sum $2c$ and four suitable pairs complete them to a
		forbidden twelve-term zero sum.  Hence $t\ne0$.  By
		Lemma~\ref{lem:f2-geometry}(v), every quotient target is represented by an
		even subset of $O$ of size $0$, $2$, or $4$.  We then require respectively
		$5$, $4$, or $3$ divided pair sums, and all these restricted sumsets of $V$
		are equal to $V$.
		
		Now suppose $|A|=19$.  Then $|O|\in\{3,5,7\}$.
		
		If $|O|=7$, use four unpaired terms and three pairs.  If $|O|=5$, use either
		two or four unpaired terms and respectively four or three pairs.  The
		necessary quotient and carry sumsets are full by
		Lemma~\ref{lem:f2-geometry}.
		
		It remains to consider $|O|=3$.  Write $O=\{u_1,u_2,u_3\}$.  The even subset
		sums form a two-dimensional subgroup
		\[
		H=\{0,u_1+u_2,u_1+u_3,u_2+u_3\},
		\]
		and the odd subset sums form the other coset $P$ of $H$.  There are eight
		pairs, so $B=V$.  Targets whose quotient lies in $H$ are obtained using zero
		or two unpaired terms, followed by five or four complete pairs.
		
		Let now $\bar x\notin H$.  If some pair has common quotient residue
		$\rho\in\bar x+P$, select one term from that pair and an odd subset
		$Q\subset O$ satisfying $\rho+\sigma(Q)=\bar x$.  If $|Q|=1$, use four
		complete pairs; if $|Q|=3$, use three.  After the broken pair is deleted, the
		remaining divided sums form a seven-set, whose three- and four-term sumsets
		are full.
		
		If no such pair exists, then every pair residue and every unpaired residue
		lies in the same affine plane $P$.  Lemma~\ref{lem:affine-plane-4} gives a
		four-term zero sum, contrary to the hypothesis.
	\end{proof}
	
	\begin{theorem}\label{thm:c4-layer9}
		Let $A$ be a sequence over $C_4^3$ such that
		\[
		|A|=18,
		\qquad
		N_4(A)=N_{12}(A)=0.
		\]
		Then
		\[
		\Sigma_9(A)=C_4^3.
		\]
	\end{theorem}
	
	\begin{proof}
		Pair equal residues modulo $2$.  Let $B$ be the divided pair-sum set and $O$
		the unpaired residue set.  Since $B\subset V$, one has
		$|O|\in\{2,4,6,8\}$.  Fix $x\in C_4^3$ and write $\bar x\in V$ for its
		quotient residue.
		
		If $|O|=6$, then $|B|=6$.  Choose three unpaired residues with sum $\bar x$,
		and choose three pair sums to absorb the carry.  Both choices are possible by
		Lemma~\ref{lem:f2-geometry}(ii).
		
		If $|O|=4$, then $|B|=7$.  Put $t=\sigma(O)$.  If $t=0$, the four unpaired
		terms together with four suitable pairs form a forbidden twelve-term zero
		sum.  Thus $t\ne0$.  The odd subset sums of $O$, of sizes one and three,
		exhaust $V$.  Use respectively four or three complete pairs.
		
		Suppose $|O|=8$.  Then $O=V$ and $|B|=5$.  By
		Lemma~\ref{lem:f2-geometry}(i),
		\[
		\Sigma_3(B)=V\setminus\{d\}
		\]
		for a unique $d\in V$.  Choose a three-subset $Q\subset O$ with quotient sum
		$\bar x$.  Unless the required carry equals $d$, three pair sums complete the
		construction.
		
		Assume that every such $Q$ requires the exceptional carry $d$.  The triples
		of sum $\bar x$ are obtained by deleting $\bar x$ from the seven affine
		planes through $\bar x$.  It follows that the sums in $C_4^3$ of the four
		lifted terms in all these planes are equal to a common value $C$.  Let $T$ be
		the sum of all eight lifted unpaired terms.  In the sum of the seven plane
		sums, the point $\bar x$ is counted seven times and every other point three
		times.  Modulo $4$, both multiplicities are $-1$, so $-T=-C$, and hence
		$T=C$.  The complement of each plane is therefore a four-term zero sum, a
		contradiction.
		
		It remains to treat $|O|=2$.  Then $B=V$.  Write $O=\{u,v\}$ and
		$t=u+v\ne0$.  Let $\rho_i\in V$ be the common quotient residue of the $i$th
		pair and put
		\[
		\Delta=\{0\}\cup\{\rho_i+\rho_j:i\ne j\}.
		\]
		If
		\begin{equation}\label{eq:c4-covering}
			(u+\Delta)\cup(v+\Delta)=V,
		\end{equation}
		then the proof is immediate.  Indeed, for quotient residues $u$ and $v$, use
		one unpaired term and four complete pairs.  For a residue
		$u+\rho_i+\rho_j$, use the $u$-term, one term from each of the two indicated
		pairs, and three complete pairs among the remaining six.  The divided sums of
		those six pairs have full three-term sumset.  The same argument applies with
		$v$.
		
		We show that failure of \eqref{eq:c4-covering} contradicts $N_4(A)=0$.  Let
		$R$ be the support of the multiset $\{\rho_1,\ldots,\rho_8\}$.  No residue
		occurs four times: otherwise eight terms of $A$ lie above one quotient class,
		and, after translation and division by $2$, Theorem~\ref{thm:sidorenko}
		produces a four-term zero sum.  Hence $|R|\ge3$.
		
		If $|R|\ge5$, then $\Delta=V$.  If $|R|=4$ and $R$ is not an affine plane,
		its six pairwise sums are distinct, so $|\Delta|=7$; the two translates by
		$u$ and $v$ cover $V$.  Thus failure is possible only when $|R|=3$, or when
		$R$ is an affine plane.  In both cases $\Delta=H$ for a two-dimensional
		subgroup $H\le V$, and failure implies $t\in H$.  The pair residues lie in
		one $H$-coset $P$, while $u$ and $v$ lie in one $H$-coset $Q$.  If $P=Q$,
		all terms of $A$ lie above one affine plane, contrary to
		Lemma~\ref{lem:affine-plane-4}.  Hence $P\ne Q$.
		
		If $|R|=4$, then $R=P$.  Quotient targets in $Q$ are already covered by
		\eqref{eq:c4-covering}.  For a target $y\in P$, choose a pair with
		$t+\rho_i=y$, take both unpaired terms and one term from that pair, and use
		three complete pairs among the remaining seven.  The carry is arbitrary.
		
		Finally suppose $|R|=3$.  The three multiplicities are $3,3,2$.  Let $r^*$ be
		the missing point of the affine plane $P$.  The only quotient residue not yet
		covered is $y^*=t+r^*$.  Since translation by nonzero $t\in H$ has no fixed
		point on $P$, one has $y^*\in R$.  Choose three pair indices whose common
		residues sum to $y^*$.  Every pair index belongs to at least one such triple:
		an index of residue $y^*$ may be combined with two indices of another common
		residue, while an index of a different residue may be combined with a second
		index of that residue and one $y^*$-index.
		
		Select one term from each of these three pairs.  The remaining five divided
		pair sums have a three-term sumset of size seven.  If one broken pair consists
		of distinct terms, switching the selected term changes the carry by a
		nonzero element of $V$, so one of the two carries lies in that seven-set.
		Thus failure would force every pair in every admissible triple to be a
		doubled singleton.  Since every pair index lies in such a triple, all eight
		pairs are doubled singletons.  Their divided sums would then have the same
		support as the common quotient residues, of size three, contradicting
		$B=V$.
	\end{proof}
	
	\begin{lemma}\label{lem:fano-separation-4}
		Let $O$ be a seven-element subset of $V$, and choose a lift
		$\widetilde x\in C_4^3$ for every $x\in O$.  The sums
		\[
		\sum_{x\in L}\widetilde x,
		\]
		where $L\subset O$ ranges over the three-element zero-sum subsets, cannot all
		be equal.
	\end{lemma}
	
	\begin{proof}
		Let $a$ be the missing point of $V\setminus O$.  Compose the lifts with a
		coordinate homomorphism associated with a nonzero linear functional
		$\lambda:V\to\F_2$.
		
		Suppose first that $a=0$.  The seven nonzero points consist of three points
		with $\lambda=0$ and four with $\lambda=1$.  Write their scalar lifts modulo
		$4$ as even elements $e_1,e_2,e_3$ and odd elements
		$o_0,o_1,o_2,o_3$.  If all seven line sums equal $K$, the kernel line gives
		$e_1+e_2+e_3=K$.  Summing all line equations, and using that every point lies
		on three Fano lines, yields
		\[
		3(e_1+e_2+e_3+o_0+o_1+o_2+o_3)=7K,
		\]
		so $o_0+o_1+o_2+o_3=0$ modulo $4$.  Summing the three lines through $o_0$
		then gives $2o_0=2K$, which is impossible because $o_0$ is odd and $K$ is
		even.
		
		Suppose now that $a\ne0$, and choose $\lambda$ with $\lambda(a)=1$.  The
		surviving zero-sum triples consist of the kernel line and three further
		lines.  Summing the latter three equations gives
		\[
		2(o_1+o_2+o_3)+(e_1+e_2+e_3)=3K.
		\]
		Using the kernel-line equation gives
		$2(o_1+o_2+o_3)=2K$, again impossible because the left side is $2$ modulo
		$4$ and the right side is $0$.
	\end{proof}
	
	\begin{theorem}\label{thm:c4-one-hole}
		Let $A$ be a sequence over $C_4^3$ such that
		\[
		|A|=17,
		\qquad
		N_4(A)=N_{12}(A)=0.
		\]
		Then
		\[
		|C_4^3\setminus\Sigma_9(A)|\le1.
		\]
	\end{theorem}
	
	\begin{proof}
		Pair equal residues modulo $2$.  The divided pair sums form a subset $B$ of
		$V$, and
		\[
		|O|\in\{1,3,5,7\}.
		\]
		
		Suppose $|O|=7$.  Then $|B|=5$.  Translation reduces an arbitrary target to
		zero, since multiplication by $9$ is the identity on $C_4^3$.  For a
		zero-sum triple $L\subset O$, use its three lifted terms and three complete
		pairs.  The three-term sumset of $B$ misses a unique carry.  If every such
		attempt failed, all lifted zero-sum triples would have the same sum,
		contradicting Lemma~\ref{lem:fano-separation-4}.  Thus this branch is in fact
		full.
		
		Suppose $|O|=5$.  Then $|B|=6$.  Put $t=\sigma(O)$.  Since
		$\Sigma_3(O)=V\setminus\{t\}$, every quotient target other than $t$ is
		represented by three unpaired terms and three pairs.  For a target with
		quotient residue $t$, use all five unpaired terms and two pairs.  The
		two-term sumset of a six-set is $V\setminus\{0\}$, so at most the unique
		target requiring zero carry can fail.
		
		Suppose $|O|=3$.  Let $H$ be the plane of even subset sums of $O$, and let
		$P$ be the coset of odd subset sums.  There are seven pairs.  The four
		quotient residues in $P$ are fully covered: use one unpaired term and four
		pairs, or all three unpaired terms and three pairs.
		
		Let $R$ be the support of the common quotient residues of the seven pairs.
		Using two unpaired terms, one term from a pair of residue $r$, and three
		complete pairs among the remaining six covers every lift of every quotient
		residue in $(H\setminus\{0\})+r$.  If all pair residues lay in $P$, every
		term of $A$ would lie above one affine plane, contrary to
		Lemma~\ref{lem:affine-plane-4}.  Hence some pair residue lies in the other
		coset.  If two distinct pair residues lie there, their sets
		$(H\setminus\{0\})+r$ cover the whole coset.  If there is only one distinct
		residue $r$ there, all quotient points except possibly $r$ are covered.  To
		cover residue $r$, break one such pair and use four complete pairs from the
		remaining six divided sums.  Their four-term sumset has size seven, so at
		most one lift of $r$ can fail.
		
		Finally suppose $|O|=1$.  There are eight pairs, and their divided sums are
		all of $V$.  Let $o$ be the unpaired quotient residue, let $\rho_i$ be the
		common residue of the $i$th pair, let $R$ be the support of the multiset
		$\{\rho_i\}$, and put
		\[
		\Delta=\{0\}\cup\{\rho_i+\rho_j:i\ne j\}.
		\]
		Every lift of a quotient residue in $R\cup(o+\Delta)$ is represented.  For
		$R$, break one pair and use four complete pairs from a seven-set.  The
		residue $o$ is represented by the unpaired term and four complete pairs.  For
		$o+\rho_i+\rho_j$, use the unpaired term, two broken pairs, and three complete
		pairs from a six-set.
		
		No common residue occurs four times, so $|R|\ge3$.  If $|R|\ge5$, then
		$\Delta=V$.  Suppose $|R|=4$.  If $R$ is an affine plane $P$ with direction
		subgroup $H$, then $\Delta=H$.  If $o+H\ne P$, the two planes cover $V$; if
		$o+H=P$, all terms lie above $P$, a contradiction.  If $R$ is not an affine
		plane, put $t=\sigma(R)\ne0$.  Then
		\[
		\Delta=V\setminus\{t\},
		\qquad
		t+R=V\setminus R.
		\]
		The only possible uncovered quotient residue is $y=o+t$.  If $o\notin R$,
		then $y\in R$ and is already covered.  If $o\in R$, choose one pair from each
		of the other three support residues.  Their quotient sum is $t+o=y$.  After
		these three pairs are broken, the remaining five divided sums have a
		three-term sumset of size seven.  Hence at most one lift of $y$ can fail.
		
		If $|R|=3$, the multiplicities are $3,3,2$, and
		$R=P\setminus\{r^*\}$ for an affine plane $P$ with direction subgroup $H$.
		If $o+H=P$, all terms lie above $P$, a contradiction.  Otherwise the only
		possible uncovered quotient residue is $r^*$.  Choose one pair from each of
		the three support residues; their quotient sum is $r^*$.  The remaining five
		divided sums again leave at most one carry uncovered.  This completes the
		proof.
	\end{proof}
	
	For a dyadic integer $q\ge4$, let the following four assertions be denoted by
	$E(q)$, $P(q)$, $R(q)$, and $H(q)$:
	\begin{align*}
		E(q):\quad & |S|=6q-3\ \Longrightarrow\ N_{3q}(S)>0;\\
		P(q):\quad & |A|\in\{6q-4,6q-5\},\ N_q(A)=N_{3q}(A)=0\\
		&\hspace{42mm}\Longrightarrow\ \Sigma_{3q-2}(A)=C_q^3;\\
		R(q):\quad & |A|=6q-6,\ N_q(A)=N_{3q}(A)=0\\
		&\hspace{42mm}\Longrightarrow\ \Sigma_{3q-3}(A)=C_q^3;\\
		H(q):\quad & |A|=6q-7,\ N_q(A)=N_{3q}(A)=0\\
		&\hspace{42mm}\Longrightarrow\ |C_q^3\setminus\Sigma_{3q-3}(A)|\le1.
	\end{align*}
	Theorems~\ref{thm:c4-rank3-threshold}, \ref{thm:c4-layer10},
	\ref{thm:c4-layer9}, and \ref{thm:c4-one-hole} show that
	\begin{equation}\label{eq:four-base-profiles}
		E(4),\quad P(4),\quad R(4),\quad H(4)
	\end{equation}
	all hold.
	
	We have two immediate consequences.
	
	\begin{lemma}\label{lem:adjacent-layers}
		Assume $E(q)$ and $P(q)$.
		\begin{enumerate}[label=\textup{(\roman*)}]
			\item If $|A|=6q-4$ and $N_q(A)=N_{3q}(A)=0$, then
			\[
			\Sigma_{3q-3}(A)=\Sigma_{3q-2}(A)=\Sigma_{3q-1}(A)=C_q^3.
			\]
			\item If $|A|=6q-5$ and $N_q(A)=N_{3q}(A)=0$, then
			\[
			\Sigma_{3q-3}(A)=\Sigma_{3q-2}(A)=C_q^3.
			\]
		\end{enumerate}
	\end{lemma}
	
	\begin{proof}
		The middle layers are $P(q)$.  In the first case, append a marked term $-x$
		to $A$.  The resulting sequence has length $6q-3$, so $E(q)$ supplies a
		$3q$-term zero sum.  It must use the marked term, and deleting it gives a
		$(3q-1)$-term subsequence of $A$ with sum $x$.  Complementation in $A$ gives
		the $(3q-3)$-layer.  The second assertion follows directly from
		complementation in a sequence of length $6q-5$.
	\end{proof}
	
	\begin{lemma}\label{lem:delete-one-hole}
		Assume $H(q)$.  If $|B|=6q-6$ and $N_q(B)=N_{3q}(B)=0$, then
		\[
		|C_q^3\setminus\Sigma_{3q-2}(B)|\le1,
		\qquad
		|C_q^3\setminus\Sigma_{3q-4}(B)|\le1.
		\]
	\end{lemma}
	
	\begin{proof}
		If two distinct elements $x,y$ were missing from $\Sigma_{3q-2}(B)$, then
		for every indexed term $b\mid B$, the two distinct elements $x-b$ and $y-b$
		would be missing from $\Sigma_{3q-3}(Bb^{[-1]})$, contradicting $H(q)$.  The
		second assertion follows by complementation in $B$.
	\end{proof}
	
	We next extend the affine-plane obstruction to every dyadic exponent.
	
	\begin{lemma}\label{lem:mixed-dyadic}
		For every $m\ge2$, every sequence of length $6\cdot2^m-7$ over
		\[
		C_{2^m}^2\oplus C_{2^{m-1}}
		\]
		contains a zero-sum subsequence of length $2^m$.
	\end{lemma}
	
	\begin{proof}
		For $m=2$, this is Lemma~\ref{lem:affine-plane-4}.  Assume $m\ge3$ and use
		induction.  Reduce a sequence of length $6\cdot2^m-7$ modulo twice the group.
		The quotient has eight elements.  Since the length is odd, maximal pairing
		leaves at most seven unpaired terms and produces at least
		\[
		\frac{6\cdot2^m-14}{2}=6\cdot2^{m-1}-7
		\]
		pairs.  Their sums, divided by $2$, form a sequence over
		$C_{2^{m-1}}^2\oplus C_{2^{m-2}}$.  By induction, $2^{m-1}$ divided pair sums
		add to zero, and the corresponding pairs form a $2^m$-term zero sum.
	\end{proof}
	
	\begin{corollary}\label{cor:affine-plane-all}
		Let $q\ge2$ be a power of $2$.  If a sequence $A$ over $C_{2q}^3$ has length
		at least $12q-7$ and all its quotient residues modulo $2$ lie in one affine
		plane of $V$, then $A$ has a zero-sum subsequence of length $2q$.
	\end{corollary}
	
	\begin{proof}
		Translate the affine plane to a linear plane.  The full inverse image of that
		plane is isomorphic to $C_{2q}^2\oplus C_q$, and
		Lemma~\ref{lem:mixed-dyadic} applies.
	\end{proof}
	
	For $T\subset V$, put
	\[
	\mathcal A(T)=\{0\}\cup T\cup\Sigma_2(T).
	\]
	
	\begin{lemma}\label{lem:at-most-two}
		If $T$ spans $V$, then $|V\setminus\mathcal A(T)|\le1$.  Every missing element
		is a sum of three distinct elements of $T$.
	\end{lemma}
	
	\begin{proof}
		Choose a basis $e_1,e_2,e_3$ from $T$.  Every vector is a sum of at most three
		basis vectors, and the only vector not represented by zero, one, or two basis
		vectors is $e_1+e_2+e_3$.  Additional elements of $T$ can only reduce the
		hole set.
	\end{proof}
	
	\begin{lemma}\label{lem:two-translates}
		Let $t\ne0$.  If $\Span(T\cup\{t\})=V$, then
		\[
		\mathcal A(T)\cup\bigl(t+\mathcal A(T)\bigr)=V.
		\]
	\end{lemma}
	
	\begin{proof}
		If $T$ spans $V$, then $\mathcal A(T)$ misses at most one point, and its
		translate by nonzero $t$ misses a different point.  If $T$ spans a plane
		$H$, then $t\notin H$ and $\mathcal A(T)=H$, so the two sets are the two
		$H$-cosets.  The remaining possibilities do not span $V$ after adjoining one
		element.
	\end{proof}
	
	\begin{lemma}\label{lem:fano-separation-general}
		Let $O\subset V$ have seven elements, and choose a lift
		$z_x\in C_{2q}^3$ for every $x\in O$, where $q\ge2$ is a power of $2$.  Fix
		$y\in V$.  The sums
		\[
		\sum_{x\in Q}z_x,
		\qquad
		Q\in\binom O3,
		\quad
		\sum_{x\in Q}x=y,
		\]
		cannot all be equal.
	\end{lemma}
	
	\begin{proof}
		Choose $w\in C_{2q}^3$ with image $y$ in $V$.  Replace $O$ by $O+y$ and set
		$z'_{x+y}=z_x+w$.  A triple of quotient sum $y$ becomes a zero-sum triple,
		and all lifted triple sums are shifted by the same element $3w$.  Equality is
		therefore preserved.  Reducing the lifted configuration modulo $4$ gives the
		situation excluded by Lemma~\ref{lem:fano-separation-4}.
	\end{proof}
	
	We use repeatedly the following elementary lifting observation.  Suppose
	$A$ is a sequence over $C_{2q}^3$, with divided pair-sum sequence $B$ and
	unpaired residue set $O$.  If we select $h$ unpaired terms, one term from each
	of $b$ broken pairs, and $m$ complete pairs, then the total length is
	\[
	h+b+2m.
	\]
	Once the quotient sum of the selected single terms is prescribed, the
	remaining difference from a target is even.  Dividing by $2$ produces a
	target in the appropriate $m$-term sumset of $B$ after the $b$ broken pair
	sums have been deleted.
	
	\begin{proposition}\label{prop:lift-H}
		Assume $E(q)$, $P(q)$, $R(q)$, and $H(q)$.  Then $H(2q)$ holds.
	\end{proposition}
	
	\begin{proof}
		Let $A$ be a sequence over $C_{2q}^3$ such that
		\[
		|A|=12q-7,
		\qquad
		N_{2q}(A)=N_{6q}(A)=0.
		\]
		Pair equal quotient residues.  The divided sequence $B$ has no $q$- or
		$3q$-term zero sum.  Since $|A|$ is odd,
		$|O|\in\{1,3,5,7\}$.  We prove that the $(6q-3)$-term sumset of $A$ has at
		most one hole.
		
		If $|O|=7$, then $|B|=6q-7$.  For any prescribed quotient sum, choose three
		unpaired residues with that sum.  We need $3q-3$ complete pairs.  The profile
		$H(q)$ leaves at most one carry unavailable.  If a fixed target in
		$C_{2q}^3$ failed for every admissible triple, all associated lifted triple
		sums would be equal, contrary to Lemma~\ref{lem:fano-separation-general}.
		Thus this branch is full.
		
		If $|O|=5$, put $t=\sigma(O)$.  Every quotient target other than $t$ is
		represented by three unpaired terms and $3q-3$ complete pairs; here
		$|B|=6q-6$, so $R(q)$ applies.  For quotient target $t$, use all five
		unpaired terms and $3q-4$ complete pairs.  By
		Lemma~\ref{lem:delete-one-hole}, at most one target in this quotient fiber can
		fail.
		
		Suppose $|O|=3$.  Let $H$ be the plane of even subset sums and $P$ the affine
		plane of odd subset sums.  Targets in $P$ use one or three unpaired terms and
		respectively $3q-2$ or $3q-3$ complete pairs.  These layers are full by
		$P(q)$ and Lemma~\ref{lem:adjacent-layers}.
		
		Consider the other affine plane.  Choose two unpaired terms and break one
		pair.  A pair residue $r$ in this plane covers every quotient point in the
		plane except $r$, because the two-term sums of $O$ are $H\setminus\{0\}$.
		If no pair residue lies there, all terms reduce into one affine plane,
		contrary to Corollary~\ref{cor:affine-plane-all}.  Two distinct pair residues
		cover the whole plane.  If there is only one distinct residue $r$, all
		quotient points except $r$ are covered using two unpaired terms, one broken
		pair, and $3q-3$ complete pairs from a divided sequence of length $6q-6$.
		For quotient residue $r$, break one such pair and use $3q-2$ complete pairs;
		Lemma~\ref{lem:delete-one-hole} leaves at most one target uncovered.
		
		Finally suppose $|O|=1$, with unpaired residue $o$.  Let $R$ be the support
		of the common quotient residues of the pairs, and put $T=R+o$.  The following
		quotient targets are fully represented:
		\begin{itemize}[leftmargin=2em]
			\item $o$, using the unpaired term and $3q-2$ complete pairs;
			\item each $r\in R$, by breaking one pair and using $3q-2$ complete pairs;
			\item each $o+r_i+r_j$, by using the unpaired term, breaking two pairs, and
			using $3q-3$ complete pairs.
		\end{itemize}
		After translation by $o$, the covered quotient set contains $\mathcal A(T)$.
		If $T$ does not span $V$, all terms reduce into one affine plane, contrary to
		Corollary~\ref{cor:affine-plane-all}.  Thus $T$ spans $V$, and
		Lemma~\ref{lem:at-most-two} leaves at most one quotient point.  If such a
		point remains, it is the sum of three distinct elements of $T$.  Break the
		corresponding three pairs and use $3q-3$ complete pairs from the remaining
		length-$(6q-7)$ divided sequence.  The profile $H(q)$ leaves at most one group
		element uncovered.
	\end{proof}
	
	\begin{proposition}\label{prop:lift-R}
		Assume $E(q)$, $P(q)$, $R(q)$, and $H(q)$.  Then $R(2q)$ holds.
	\end{proposition}
	
	\begin{proof}
		Let $|A|=12q-6$ and suppose
		$N_{2q}(A)=N_{6q}(A)=0$.  The possible values of $|O|$ are
		$0,2,4,6,8$.
		
		If $|O|=0$, then $|B|=6q-3$, and $E(q)$ gives a $3q$-term zero sum in $B$,
		lifting to a forbidden $6q$-term zero sum.
		
		Suppose $|O|=2$, say $O=\{u,v\}$, and put $t=u+v\ne0$.  Let $R$ be the
		support of the pair residues and set $T=R+u$.  Using one unpaired term, one or
		two broken pairs, or both unpaired terms and one broken pair, the quotient
		constructions cover
		\[
		u+\mathcal A(T)
		\quad\text{and}\quad
		u+t+\mathcal A(T).
		\]
		The required lower layers are $3q-2$ and $3q-3$ in divided sequences of
		lengths $6q-5$ and $6q-6$; they are full by $P(q)$, $R(q)$, and
		Lemma~\ref{lem:adjacent-layers}.  If $\Span(T\cup\{t\})=V$,
		Lemma~\ref{lem:two-translates} gives every quotient target.  Otherwise all
		terms reduce into one affine plane, contrary to
		Corollary~\ref{cor:affine-plane-all}.
		
		If $|O|=4$, then $\sigma(O)\ne0$: otherwise the four unpaired terms and
		$3q-2$ suitable pairs form a forbidden $6q$-term zero sum.  By
		Lemma~\ref{lem:f2-geometry}(v), the odd subset sums of $O$, of sizes one and
		three, exhaust $V$.  Use respectively $3q-2$ and $3q-3$ complete pairs.
		
		If $|O|=6$, choose three unpaired terms with any prescribed quotient sum and
		then use $3q-3$ complete pairs from the full layer supplied by $R(q)$.
		
		If $|O|=8$, choose three unpaired terms with the prescribed quotient sum and
		use the one-hole layer $H(q)$ in the divided sequence of length $6q-7$.
		Lemma~\ref{lem:fano-separation-general} rules out simultaneous failure of all
		admissible quotient triples.
	\end{proof}
	
	\begin{proposition}\label{prop:lift-P}
		Assume $E(q)$, $P(q)$, $R(q)$, and $H(q)$.  Then $P(2q)$ holds.
	\end{proposition}
	
	\begin{proof}
		First let $|A|=12q-4$ and suppose
		$N_{2q}(A)=N_{6q}(A)=0$.  The cases $|O|=0,2$ are impossible, since then
		$|B|\ge6q-3$ and $E(q)$ gives a $3q$-term zero sum.
		
		If $|O|=4$, then $\sigma(O)\ne0$, for otherwise all four unpaired terms
		together with $3q-2$ suitable pairs form a $6q$-term zero sum.  The even
		subset sums of $O$, using zero, two, or four terms, exhaust $V$.  The
		corresponding required lower layers are $3q-1$, $3q-2$, and $3q-3$ in a
		divided sequence of length $6q-4$, and all are full by
		Lemma~\ref{lem:adjacent-layers} and $P(q)$.
		
		If $|O|=6$, every quotient target is a sum of either two or four unpaired
		residues.  Use respectively $3q-2$ or $3q-3$ complete pairs in the divided
		sequence of length $6q-5$.  If $|O|=8$, use four unpaired terms and $3q-3$
		complete pairs from $R(q)$.
		
		Now let $|A|=12q-5$.  The case $|O|=1$ is impossible by $E(q)$.  If
		$|O|=5$, every quotient target is a sum of two or four unpaired residues,
		followed by $3q-2$ or $3q-3$ complete pairs.  If $|O|=7$, use four unpaired
		terms and $3q-3$ complete pairs.
		
		It remains to treat $|O|=3$.  Let $H$ be the plane of even subset sums and
		$P$ the plane of odd subset sums.  Targets in $H$ use zero or two unpaired
		terms and the full lower layers $3q-1$ or $3q-2$.  Let $y\notin H$.  If a
		pair has common quotient residue $r\in y+P$, select one term from that pair
		and an odd subset of $O$ with quotient sum $y+r$.  According as the odd
		subset has size one or three, use $3q-2$ or $3q-3$ complete pairs from the
		remaining divided sequence of length $6q-5$.  If no such pair exists, every
		pair residue and every unpaired residue lies in the same affine plane $P$,
		contrary to Corollary~\ref{cor:affine-plane-all}.
	\end{proof}
	
	\begin{lemma}\label{lem:dyadic-no-q}
		Let $q$ be a power of $2$, and let $S$ be a sequence over $C_q^3$ of length
		$6q-3$ containing no zero-sum subsequence of length $3q$.  Then $N_q(S)=0$.
	\end{lemma}
	
	\begin{proof}
		Suppose $Q_1$ is a $q$-term zero sum.  The complement has length $5q-3$.
		Choose $4q-3$ terms from it and apply Lemma~\ref{lem:olson-affine}.  A
		$2q$-term zero sum together with $Q_1$ gives a $3q$-term zero sum, while a
		$3q$-term zero sum is forbidden.  Hence there is a second disjoint $q$-term
		zero sum $Q_2$.  The remaining sequence has length $4q-3$; another
		application of Lemma~\ref{lem:olson-affine} gives a $q$-, $2q$-, or
		$3q$-term zero sum, and every case produces a $3q$-term zero sum in $S$.
	\end{proof}
	
	\begin{proposition}\label{prop:lift-E}
		Assume $E(q)$, $P(q)$, $R(q)$, and $H(q)$.  Then $E(2q)$ holds.
	\end{proposition}
	
	\begin{proof}
		Let $S$ be a sequence over $C_{2q}^3$ with $|S|=12q-3$, and suppose
		$N_{6q}(S)=0$.  By Lemma~\ref{lem:dyadic-no-q}, applied at exponent $2q$,
		one has $N_{2q}(S)=0$.
		
		Since multiplication by $3$ is an automorphism of $C_{2q}^3$, translate $S$
		so that $\sigma(S)=0$.  Translation preserves all zero-sum lengths divisible
		by $2q$.  After pairing, $\sigma(O)=0$ and
		$|O|\in\{1,3,5,7\}$.
		
		If $|O|=1$ or $3$, then $|B|\ge6q-3$, and $E(q)$ gives a $3q$-term zero sum,
		lifting to a $6q$-term zero sum.
		
		If $|O|=5$, then $0\in O$ and $O\setminus\{0\}$ is a four-element zero-sum
		set.  The four corresponding lifted terms have sum $2c$ for some
		$c\in C_q^3$.  Here $|B|=6q-4$, and $P(q)$ supplies $3q-2$ terms of $B$ with
		sum $-c$.  Their pairs and the four unpaired terms form a $6q$-term zero sum.
		
		If $|O|=7$, then $O=V\setminus\{0\}$.  Choose a four-element zero-sum subset
		of $O$ and argue identically, using the length-$(6q-5)$ case of $P(q)$.
	\end{proof}
	
	\begin{theorem}\label{thm:simultaneous-lifting}
		For every dyadic $q\ge4$,
		\[
		E(q),P(q),R(q),H(q)
		\quad\Longrightarrow\quad
		E(2q),P(2q),R(2q),H(2q).
		\]
	\end{theorem}
	
	\begin{proof}
		Combine Propositions~\ref{prop:lift-H}, \ref{prop:lift-R},
		\ref{prop:lift-P}, and \ref{prop:lift-E}.
	\end{proof}
	
	\begin{proof}[Proof of Theorem~\ref{thm:intro-dyadic}]
		For $n=1$, Theorem~\ref{thm:sidorenko} gives
		$\mathsf s_6(C_2^3)=9$.  For $n\ge2$, the four base profiles in
		\eqref{eq:four-base-profiles} hold, and
		Theorem~\ref{thm:simultaneous-lifting} propagates $E(q)$ through all dyadic
		exponents.  Hence
		\[
		\mathsf s_{3q}(C_q^3)\le6q-3.
		\]
		The reverse inequality is \eqref{eq:standard-lower}, since
		$D(C_q^3)=3q-2$.  The values for $k\ge4$ are given by
		Theorem~\ref{thm:known-rank-three}(i).
	\end{proof}
	
	\section{Proof of Theorem \ref{thm:intro-ternary}}\label{sec:ternary-rank-three}
	
	Let $q=3^n$ and $G=C_q^3$.  The values at multipliers $3$ and $5$ are known,
	as are all multipliers at least $6$.  We settle the remaining multiplier
	$4$.
	
	\begin{proof}[Proof of Theorem~\ref{thm:intro-ternary}]
		Let $q=3^n$.  The standard lower bound is
		\[
		4q+D(C_q^3)-1=4q+(3q-2)-1=7q-3.
		\]
		Assume that $S$ is a sequence of length $7q-3$ with no $4q$-term zero sum,
		and set
		\[
		A_j=N_{jq}(S),
		\qquad A_0=1.
		\]
		
		\emph{Step 1: $A_1=0$.}
		If $Q$ were a $q$-term zero sum, the complement would have length $6q-3$.
		The known exact multiplier-three value would give a disjoint $3q$-term zero
		sum, and their union would have length $4q$.
		
		\emph{Step 2: $A_6=0$.}
		Suppose that $U\mid S$ were a $6q$-term zero sum.  It contains no zero-sum
		subsequence of length $q$ or $4q$.  A $2q$-term zero sum would have a
		$4q$-term zero-sum complement in $U$, while a $5q$-term zero sum would have a
		$q$-term zero-sum complement.  Thus $U$ has no zero-sum subsequence of any
		length in $\{q,2q,4q,5q\}$. This contradicts Lemma~\ref{lem:ternary-four-length}, since
		$6q\ge\frac{17}{3}q-3$.  Hence
		\begin{equation}\label{eq:ternary-vanishing}
			A_1=A_4=A_6=0.
		\end{equation}
		
		\emph{Step 3:}
		Apply Theorem~\ref{thm:weighted-moments} with $p=3$, $m=1$, $t=3$, and
		$\alpha=0$.  The required length is
		\[
		\frac{2q}{3}+3q-1+(3q-2)=\frac{20q}{3}-3\le7q-3.
		\]
		Since $q\equiv3\pmod6$, one has $2^q\equiv-1\pmod9$.  Put $a=A_2$, $b=A_3$, and $c=A_5$.
		Combining with \eqref{eq:ternary-vanishing}, we have
		\begin{align}
			1+a-b-c&\equiv0\pmod9,\label{eq:t9-0}\\
			2a-3b-5c&\equiv0\pmod9,\label{eq:t9-1}\\
			4a-25c&\equiv0\pmod9.\label{eq:t9-2}
		\end{align}
		Equation \eqref{eq:t9-2} is equivalent to $2a+c\equiv0\pmod9$.
		Substituting this into \eqref{eq:t9-1} gives $b\equiv a\pmod3$, and reducing
		\eqref{eq:t9-0} modulo $3$ yields
		\begin{equation}\label{eq:ternary-residues}
			a\equiv b\equiv c\equiv1\pmod3.
		\end{equation}
		
		\emph{Step 4:}
		For every subsequence $T\mid S$ of length $4q-3$,
		Lemma~\ref{lem:olson-affine} gives
		\[
		1-N_q(T)+N_{2q}(T)-N_{3q}(T)\equiv0\pmod3.
		\]
		Sum this congruence over all such $T$.  A fixed $jq$-term zero sum of $S$ is
		counted
		\[
		\binom{7q-3-jq}{4q-3-jq}
		\]
		times.  Lucas' theorem gives
		\[
		\binom{7q-3}{4q-3}\equiv2,
		\quad
		\binom{6q-3}{3q-3}\equiv1,
		\quad
		\binom{5q-3}{2q-3}\equiv1,
		\quad
		\binom{4q-3}{q-3}\equiv1
		\pmod3.
		\]
		Consequently,
		\[
		2-A_1+A_2-A_3\equiv0\pmod3.
		\]
		By $A_1=0$ and \eqref{eq:ternary-residues}, the left side is congruent to
		$2$, a contradiction.  Thus every sequence of length $7q-3$ contains a
		$4q$-term zero sum.
	\end{proof}
	
	\begin{proof}[Proof of Theorem~\ref{thm:intro-rank-three}]
		For powers of $2$, use Theorem~\ref{thm:intro-dyadic}.  For powers of $3$,
		Theorem~\ref{thm:intro-ternary} supplies the only missing multiplier, while
		Theorem~\ref{thm:known-rank-three}(ii) supplies the others.  For primes at
		least $5$, use Theorem~\ref{thm:known-rank-three}(iii).  The lower bound is
		\eqref{eq:homocyclic-lower}.
	\end{proof}
	
	\medskip
	\noindent\textbf{Statement of AI.} The authors are completely responsible for the proofs.
	ChatGPT was used to assist with routine calculations and checks, language
	editing, and reorganization of the manuscript.
	
	\medskip
	\noindent\textbf{Acknowledgment.} This work was supported by Guangdong
	Basic and Applied Basic Research Foundation Grant No.~2024A1515012564.

\end{document}